\documentclass[12pt,fullpage]{amsart}

\usepackage{amssymb,amsmath,amsthm,amsfonts}
\usepackage[mathscr]{euscript}
\usepackage{dsfont}
\usepackage{mathrsfs}
\usepackage{esint}
\usepackage{cases}
\usepackage{graphicx}
\usepackage{subfig}
\usepackage{enumitem}
\usepackage{geometry}
\usepackage{comment}
\usepackage{mathdots}
\renewcommand{\leq}{\leqslant}

\renewcommand{\geq}{\geqslant}

\usepackage{color}
\definecolor{citation}{rgb}{0.2,0.5,0.2}
\definecolor{formula}{rgb}{0.1,0.2,0.5}
\definecolor{url}{rgb}{0,0.2,0.7}
\usepackage[colorlinks=true,linkcolor=formula,urlcolor=url,citecolor=citation]{hyperref}

\newtheoremstyle%
   {olivier}
   {1.5ex}
   {1.5ex}
   {\sl}
   {11.5pt}
   {\bf\sc}   
   {.~---}
   {1ex}
   {}

\newtheoremstyle%
   {oliv0}
   {1.5ex}
   {1.5ex}
   {\sl}
   {11.5pt}
   {\bf\sc}   
   {~---}
   {1ex}
   {}

\newtheoremstyle%
   {clovis}
   {1.5ex}
   {1.5ex}
   {\normalfont}
   {11.5pt}
   {\sl}               
   {.~---}
   {1ex}
   {}

\newtheorem{theo}{Theorem}[section]

\newtheorem{lemma}[theo]{Lemma}

\theoremstyle{definition}
\newtheorem{de}[theo]{Definition}
\theoremstyle{fact}

\theoremstyle{remark}
\newtheorem{remark}[theo]{Remark}

\numberwithin{equation}{section}

\def\R {\mathbb{R}}

\def\N {\mathbb{N}}

\def\eps{\varepsilon}

\newlength{\defbaselineskip}
\newcommand{\setlinespacing}[1]
           {\setlength{\baselineskip}{#1 \defbaselineskip}}

\allowdisplaybreaks[4]

\begin{document}

\title[How smooth are restrictions of Besov functions?]{How smooth are restrictions of Besov functions?}

\author[Julien Brasseur]{Julien Brasseur}
\email{julienbrasseur@wanadoo.fr}

\begin{abstract}
In a previous work, we showed that Besov spaces do not enjoy the restriction property unless $q\leq p$. Specifically, we proved that if $p<q$, then it is always possible to construct a function $f\in B_{p,q}^s(\mathbb{R}^N)$ such that $f(\cdot,y)\notin B_{p,q}^s(\mathbb{R}^d)$ for a.e. $y\in \mathbb{R}^{N-d}$, while this ``pathology" does not happen if $q\leq p$. We showed that the partial maps belong, in fact, to the Besov space of generalised smoothness $B_{p,q}^{(s,\Psi)}(\mathbb{R}^d)$ provided the function $\Psi$ satisfies a simple summability condition involving $p$ and $q$. This short note completes the picture by showing that this characterisation is sharp.
\end{abstract}



\maketitle

\tableofcontents

\section{Introduction}

In a previous work \cite{Brasseur} we showed, concomitantly to \cite{MRS}, that Besov spaces exhibit a striking property when looking at their partial maps. Precisely, letting $1\leq d<N$ and 
\begin{align}
 0<p,q\leq \infty,\quad s>\sigma_p,\quad \sigma_p=N\hspace{0.1em}\bigg(\frac{1}{p}-1\bigg)_+, 
\label{sigmap}
\end{align}
we proved that, when $p<q$, there exists functions $f\in B_{p,q}^s(\R^N)$ such that
$$ f(\cdot,y)\notin B_{p,\infty}^s(\R^d) \, \text{ for a.e. }y\in\R^{N-d}, $$
(see \cite[Theorem 1.3]{Brasseur}) while, for $q\leq p$, we have
$$ f\in B_{p,q}^s(\R^N) \Longrightarrow f(\cdot,y)\in B_{p,q}^s(\R^d) \,\text{ for a.e. }y\in\R^{N-d}, $$
(see \cite[Proposition 5.1]{Brasseur} or \cite[Proposition 6.10]{MRS}).

This violation of what we called the ``restriction property" is quite surprising for at least two reasons: first, most function spaces of this type (such as Sobolev spaces, fractional Sobolev spaces, Triebel-Lizorkin spaces and so on) are known to satisfy the ``restriction property"; and, second, because the paramater $q$ usually plays little to no role in the properties of Besov spaces (we will come back to this later on, see Remark~\ref{RQ:q}).

In \cite{Brasseur}, we showed that, when $p<q$, the partial maps belong to a scale of functions, intercalated between $B_{p,q}^{s-\eps}$ and $B_{p,q}^s$, known as ``Besov spaces of generalised smoothness'' and denoted by $B_{p,q}^{(s,\Psi)}$. Functions of this type still have $s$ as dominant smoothness but the latter is ``modulated" by a function $\Psi$ (typically of $\log$-type), which allows encoding more general kinds of smoothness. Specifically, we established that the partial maps of any $f\in B_{p,q}^s(\R^N)$ belong to $B_{p,p}^{(s,\Psi)}(\R^d)$, provided
\begin{align}
\bigg(\sum_{j=0}^\infty\Psi(2^{-j})^\varkappa\bigg)^\frac{1}{\varkappa}<\infty \,\, \text{ where } \,\, \frac{1}{\varkappa}=\frac{1}{p}-\frac{1}{q}. \label{cond:psi}
\end{align}
(For a discussion on the relevance of this scale in physical problems or spaces on fractals we refer to \cite{Brasseur}, where we also gave a brief historical account and references.)

Let us summarise the above more precisely:
\begin{theo}\label{TH:old}
Let $N\geq2$, $1\leq d<N$, $0<p<q\leq\infty$, $s>\sigma_p$ and let $\Psi$ be a slowly varying function satisfying \eqref{cond:psi}. Suppose that $f\in B_{p,q}^s(\R^N)$. Then,
$$ f(\cdot,y)\in B_{p,p}^{(s,\Psi)}(\R^d)\quad\text{for a.e. }y\in\R^{N-d}. $$
\end{theo}
\begin{remark}
This result was initially stated for admissible functions in the sense of Edmunds and Triebel \cite{ET1,ET2}, but the proof generalises to slowly varying functions without difficulty. (Here ``slowly varying'' is meant in the sense of Karamata, see Definition~\ref{DE:slow} below.)
\end{remark}
Nevertheless, it is natural to ask:
\begin{center}
\emph{Does this precisely ``measure" the loss of regularity?}
\end{center}
We gave only a partial result (see \cite[Theorem 1.7]{Brasseur}) suggesting that this might be the case but, so far, the question has remained open. This note fixes this gap by showing that, indeed, the above fully characterises the smoothness of the partial maps.


The main result of this note is the following:
\begin{theo}\label{TH:MAIN}
Let $N\geq2$, $1\leq d<N$, $0<p<q\leq\infty$, $s>\sigma_p$ and let $\Psi$ be a slowly varying function that does not satisfy \eqref{cond:psi}. Then, there exists a function $f\in B_{p,q}^s(\R^N)$ such that
$$ f(\cdot,y)\notin B_{p,\infty}^{(s,\Psi)}(\R^d)\quad\text{for a.e. }y\in\R^{N-d}. $$
\end{theo}
Hence we arrive at a sharp characterisation of partial maps of Besov functions. That is, for all parameters $s$, $p$ and $q$ as in \eqref{sigmap}, the ``compensated" restriction property:
$$ f\in B_{p,q}^s(\R^N) \Longrightarrow f(\cdot,y)\in B_{p,\min(p,q)}^{(s,\Psi)}(\R^d) \text{ for a.e. }y\in\R^{N-d}, $$
holds if, and only if, \eqref{cond:psi} holds. (Here, we implicitely extend \eqref{cond:psi} to the case $q\leq p$ by replacing the right-hand side by its positive part with the usual convention $1/\infty = 0$, so that $\Psi$ is merely required to be bounded when $q\leq p$ and the corresponding Besov space of generalised smoothness boils down to $B_{p,q}^s(\R^d)$.) 
\begin{remark}
Even though the above results are stated for slowly varying functions $\Psi$, our construction does not require this assumption. This assumption is, in fact, made to ensure that the Besov spaces of generalised smoothness enjoy ``nice" representations (that is, equivalent quasi-norms via quarkonial/subatomic decompositions, Littlewood-Paley and finite differences). This is the only place where an assumption on $\Psi$ comes into play. The important thing to note is that \emph{our main result holds whenever an assumption on $\Psi$ is made so that Besov spaces of generalised smoothness enjoy the above-mentionned representations}. To be precise, (the proof of) Theorem~\ref{TH:old} uses only Littlewood-Paley and either quarkonial/subatomic decompositions or representation via differences, while (the proof of) Theorem~\ref{TH:MAIN} requires only representation via differences.
\end{remark}
\begin{remark}\label{RQ:q}
We briefly mentioned that, usually, the parameter $q$ plays no role on the properties of Besov spaces. There are, however, a few results where it does. It is not our intention to make a comprehensive list of these, but we want to point out two of them which are of some relevance. Both hold in the borderline case $0<p<1$ only. The first result in this direction is due to Johnsen who showed (see \cite[Theorem 1.2]{Johnsen}) that the distributional trace operator of the critical Besov space $B_{p,q}^{n/p-n+1}(\R^n)$ enjoys different properties and codomains depending on whether $q\leq p$ or $p<q$. The second result we would like to mention is due to Caetano and Haroske who showed (see \cite[Corollary 3.17]{Caetano}) that traces on fractal sets $\Gamma$ of spaces of the type of $B_{p,q}^{(s,\Psi)}$ embed in $L^p(\Gamma)$ if a certain sequence $\sigma$ depending on $\Psi$ satisfies some conditions which differ according to whether $q\leq p$ or $p<q$ and, in the latter case, $\sigma$ is required to belong to the sequence space $\ell^\varkappa(\N)$. 
Although these results differ in nature from our own, the family resemblance is appealing enough to be worth mentioning.
\end{remark}

\section{Notation and definitions}

As usual, $\R$ denotes the set of all real numbers and $\N$ the set of all nonnegative integers $\{0,1,2,...\}$. 
The $N$-dimensional real Euclidian space will be denoted by $\R^N$. Similarly, $\N^N$ denotes the lattice of points $(m_1,\cdots,m_N)\in\R^N$ with $m_j\in\N$ for all $j\in\{1,\cdots,N\}$. We will sometimes make use of the approximatively-less-than symbol ``$\lesssim$", that is we write $a\lesssim b$ for $a\leq C\,b$ where $C>0$ is a constant independent of $a$ and $b$. Similarly, $a\gtrsim b$ means that $b\lesssim a$. Also, we denote by $\lfloor x\rfloor$ the integral part of $x\in\R$ and by $x_+$ its positive part. 

Let $f$ be a function in $\mathbb{R}^N$. Given $M\in\mathbb{N}$ and $h\in\mathbb{R}^N$, let
$$ \Delta_h^Mf(x)=\sum_{j=0}^M(-1)^{M-j}\binom{M}{j}f(x+hj). $$

We will define the classical Besov spaces and their generalised counterpart all at once. But before doing so, we recall the definition of slowly varying functions.
\begin{de}\label{DE:slow}
Let $\Psi$ be a positive, measurable function defined on the interval $(0,1]$. We say that $\Psi$ is \emph{slowly varying} if it satisfies
$$ \lim_{t\to\infty}\,\frac{\Psi(r t)}{\Psi(t)}=1 \quad \text{for all }r\in(0,1]. $$
\end{de}
\begin{remark}
These functions were introduced by Karamata in the mid thirties in two seminal papers (see \cite{Karamata1,Karamata2}). An extensive study of their properties can be found in the monograph of Bingham, Goldie and Teugels \cite{Bingham}, where various examples are given (see also \cite{Gut}).
\end{remark}
\begin{de}\label{G:BESOV}
Let $0<p,q\leq\infty$, $s>\sigma_p$ and let $\Psi$ be a slowly varying function. Let $M\in\N$ with $M>s$. The \emph{Besov space of generalised smoothness} $B_{p,q}^{(s,\Psi)}(\mathbb{R}^N)$ consists of all functions $f\in L^p(\mathbb{R}^N)$ such that
\begin{align*}
[f]_{B_{p,q}^{(s,\Psi)}(\mathbb{R}^N)}=\left(\int_{0}^1\,\sup_{|h|\leq t}\|\Delta_h^Mf\|_{L^p(\R^N)}^q\frac{\Psi(t)^q}{t^{1+sq}}\,\mathrm{d}t\right)^{1/q}<\infty, 
\end{align*}
which, in the case $q=\infty$, is to be understood as
$$ [f]_{B_{p,\infty}^{(s,\Psi)}(\mathbb{R}^N)}=\sup_{0<t\leq1}t^{-s}\Psi(t)\sup_{|h|\leq t}\|\Delta_h^Mf\|_{L^p(\mathbb{R}^N)}<\infty. $$
The space $B_{p,q}^{(s,\Psi)}(\mathbb{R}^N)$ is naturally endowed with the quasi-norm
\begin{align}
\|f\|_{B_{p,q}^{(s,\Psi)}(\mathbb{R}^N)}=\|f\|_{L^p(\mathbb{R}^N)}+[f]_{B_{p,q}^{(s,\Psi)}(\mathbb{R}^N)}. \label{G:NORM:FD}
\end{align}
When $\Psi\equiv1$, the space $B_{p,q}^{(s,1)}(\R^N)$ is simply called \emph{Besov space} and denoted by $B_{p,q}^s(\R^N)$.
\end{de}
\begin{remark}\label{G:M}
Different choices of $M$ in \eqref{G:NORM:FD} yield equivalent quasi-norms.
\end{remark}

Besov spaces of generalised smoothness have a rather long and rich history. We refer to our previous paper \cite{Brasseur} (and references therein) for further details on the matter.

\section{Preliminary lemmata}

We first state an easy (yet important) lemma.
\begin{lemma}\label{LE}
Let $(u_j)_{j\in\N}$ be a positive sequence with divergent series. Then, the series
$$ \sum_{j=1}^\infty\frac{u_j}{(u_1+\cdots+u_j)^m}, $$
is convergent for all $m>1$ and divergent otherwise.
\end{lemma}
\begin{proof}
The case $m\leq1$ being already well-known (see e.g. \cite{Ash} or \cite[p.79]{Rudin}), it remains to prove the result for $m>1$. Set $U_j=u_1+\cdots+u_j$, for all $j\geq1$. Then, the area of the gray rectangle in Figure~\ref{FIG1} is $(U_j-U_{j-1})\times 1/U_j^m=u_j/U_j^m$.
Thus, the series $\sum u_j/U_j^m$ is controlled by the area under the curve $x\mapsto 1/x^m$ which is finite on any interval of the form $[X,\infty)$ with $X>0$ (because $m>1$). This completes the proof.
\end{proof}

\begin{figure}[!ht]
\centering
\includegraphics[scale=0.45]{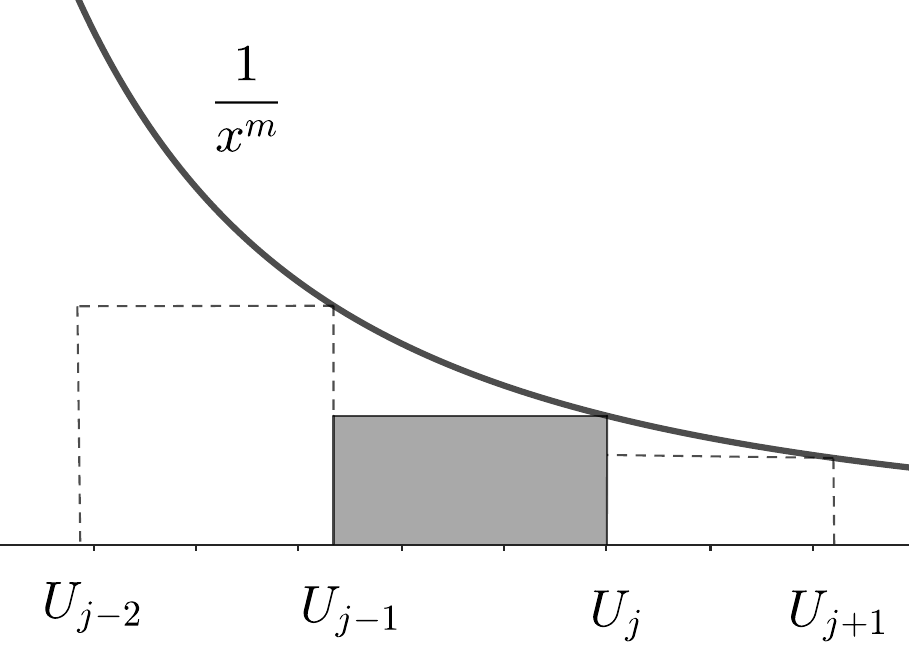}
\caption{\footnotesize Each rectangle represents a term of the series $\sum u_j/U_j^m$ and the continuous line the curve $1/x^m$. The gray rectangle has width $U_j-U_{j-1}=u_j$ and height $1/U_j^m$.}\label{FIG1}
\end{figure}


For the sake of clarity, we now state a technical result that summarises the main logic behind the construction in (the proof of) \cite[Theorem 1.7]{Brasseur}, which we shall also need here.

\begin{lemma}\label{LE:TECH}
Let $(\tau_k)_{k\in\N}\subset[0,1]$ be a sequence with divergent series and let $(\theta_k)_{k\in\N}$ be an arbitrary sequence. Let $(u_k)_{k\in\N}$ be the sequence such that $u_0=u_1=0$ and such that, for all $j\geq1$, the $\lfloor 2^j\tau_j\rfloor$ first terms of $(u_k)_{k\in\N}$ on the discrete dyadic interval $T_j=\{2^j, \cdots, 2^{j+1}-1\}$ have value $\theta_j$ and the remaining terms on $T_j$ are all equal to zero. Then, there exists a rearrangement, $(u_k^*)_{k\in\N}$, of $(u_k)_{k\in\N}$, which preserves the average on each $T_j$ and such that, for all $x\in[1,2)$, there exists a countably infinite set $J_x\subset\N$ such that
$$ u_{\lfloor 2^j x\rfloor}^* = \theta_j  \quad\text{for all }j\in J_x. $$
\end{lemma}
\begin{proof}
The proof essentially coincides with the rearrangement in the proof of \cite[Lemma~3.6]{Brasseur}, up to minor changes. There, it was carried out with $\tau_j=1/j$ and $\theta_j=j$ but the rearrangement made no use of the specific form of these two sequences. Only the divergence of the series $\sum\tau_j$ and the fact that $0\leq\tau_j\leq1$ was needed. Aside from these minor modifications, the arguments given in \cite{Brasseur} adapt without difficulty.
\end{proof}

We are now in position to prove the most important result of this section.

\begin{lemma}\label{LE:main}
Let $0<p<q\leq\infty$ and let $\Psi$ be a positive function defined on the interval $(0,1]$ which does not satisfy \eqref{cond:psi}. Then, there exists a positive sequence $(\lambda_{j,k})_{j,k\in\N}$ such that
$$ \bigg(\sum_{j=0}^\infty\bigg(\sum_{k=0}^\infty\lambda_{j,k}^p\bigg)^{\frac{q}{p}}\bigg)^\frac{1}{q}<\infty, $$
(with the usual modification if $q=\infty$) and
$$ \sup_{j=0}\, 2^{\frac{j}{p}}\lambda_{j,\lfloor 2^jx\rfloor}\Psi(2^{-j})=\infty \quad\text{for all }x\in[1,2). $$
\end{lemma}
\begin{remark}
By ``usual modification if $q=\infty$" we mean that, when $q=\infty$, the quasi-norm $\left\|\cdot\right\|_{\ell^q(\N)}=(\sum_{j\geq0}\left|\cdot\right|^q)^{1/q}$ is to be replaced by the $\sup$-norm $\left\|\cdot\right\|_{\ell^\infty(\N)}=\sup_{j\geq0}\left|\cdot\right|$.
\end{remark}
\begin{proof}
Since we already settled the case $q=\infty$ in \cite[Lemma 3.8]{Brasseur} (note that the result there is stated for admissible functions $\Psi$ rather than merely positive ones, but the proof for $q=\infty$ makes no use of that assumption, and so adapts to our setting without difficulty), we only have to consider the case $q<\infty$. To do so, we will follow the same kind of construction as we used in \cite{Brasseur} and use Lemma~\ref{LE} to get rid of the unnecessary assumptions we had made. But first, we need to introduce some notation. For $j\in\N$ and $m\geq1$, we set
$$ \Gamma_{j,m}=\frac{\Psi(2^{-j})^\varkappa}{S_j^m},\quad S_j=\sum_{k=1}^j\Psi(2^{-k})^\varkappa. $$
Also, for any $j\in\N$, we let
$$ T_j=\{2^j,\cdots,2^{j+1}-1\}, $$
and
$$ 0<L<1-\frac{p}{q}. $$
(Note that $L$ is well-defined because $p<q$.)
Let $(\Lambda_k)_{k\in\N}$ be such that $\Lambda_0=\Lambda_1=0$ and such that, for any $j\geq1$, the $\lfloor2^j\Gamma_{j,1}\rfloor$ first terms of the sequence $(\Lambda_k)_{k\in\N}$ on the discrete dyadic interval $T_j$ have value $S_j^{L}/\Psi(2^{-j})^{p}$ and the remaining terms on $T_j$ are all equal to zero.

Then, for any $j\geq1$, we have
$$ \frac{1}{2^j}\sum_{k\in T_j}\Lambda_k=\frac{1}{2^j}\left(\frac{S_j^{L}}{\Psi(2^{-j})^{p}}+\cdots+\frac{S_j^{L}}{\Psi(2^{-j})^{p}}+0+\cdots+0\right)=\frac{\lfloor2^j\Gamma_{j,1}\rfloor S_j^{L}}{2^j\Psi(2^{-j})^{p}}\leq \frac{\Gamma_{j,1}S_j^{L}}{\Psi(2^{-j})^p}. $$
On the other hand, since $\frac{q}{p}(\varkappa-p)=\varkappa$, we have
$$ \bigg(\frac{\Gamma_{j,1}S_j^{L}}{\Psi(2^{-j})^p}\bigg)^\frac{q}{p}=\bigg(\frac{\Psi(2^{-j})^{\varkappa-p}}{S_j^{1-L}}\bigg)^\frac{q}{p}=\Gamma_{j,\frac{q}{p}(1-L)}, $$
and since $\frac{q}{p}(1-L)>1$, it follows from Lemma~\ref{LE} that $(\Gamma_{j,\frac{q}{p}(1-L)})_{j\in\N}$ is summable. Hence
\begin{align}
\bigg(\sum_{j=0}^\infty\bigg(\frac{1}{2^j}\sum_{k\in T_j}\Lambda_k\bigg)^{\frac{q}{p}}\bigg)^{\frac{1}{q}}<\infty. \label{partie:cv}
\end{align}
Now, by Lemma~\ref{LE:TECH}, we can rearrange $(\Lambda_j)_{j\in\N}$ so to obtain a sequence $(\Lambda_j^*)_{j\in\N}$ satisfying \eqref{partie:cv} (with $\Lambda_j^*$ in place of $\Lambda_j$) and such that, for any $x\in[1,2)$, there exists a countably infinite set $J_x\subset\N$ for which we have
$$ \Lambda_{\lfloor2^jx\rfloor}^*=\frac{1}{\Psi(2^{-j})^p}\bigg(\sum_{k=1}^j\Psi(2^{-k})^\varkappa\bigg)^L \quad\text{for all }j\in J_x. $$
The conclusion now follows by setting
$$
\lambda_{j,k}=\left\{
\begin{array}{cl}
2^{-\frac{j}{p}}(\Lambda_j^*)^\frac{1}{p} & \text{if }k\in T_j, \\
0 & \text{otherwise.}
\end{array}
\right.
$$
Indeed, on the one hand, we have
$$ \bigg(\sum_{j=0}^\infty\bigg(\sum_{k=0}^\infty\lambda_{j,k}^p\bigg)^{\frac{q}{p}}\bigg)^\frac{1}{q}=\bigg(\sum_{j=0}^\infty\bigg(\frac{1}{2^j}\sum_{k\in T_j}\Lambda_k^*\bigg)^{\frac{q}{p}}\bigg)^{\frac{1}{q}}<\infty, $$
while, on the other hand, we have
$$ \sup_{j\in\N}\, 2^{\frac{j}{p}}\lambda_{j,\lfloor 2^jx\rfloor}\Psi(2^{-j})=\sup_{j\in\N}\, (\Lambda_{\lfloor 2^jx\rfloor}^*)^\frac{1}{p}\Psi(2^{-j})\geq \sup_{j\in J_x} \bigg(\sum_{k=1}^j\Psi(2^{-k})^\varkappa\bigg)^{\frac{L}{p}}=\infty, $$
for all $x\in[1,2)$. This thereby completes the proof.
\end{proof}
\begin{remark}
Observe that $\Psi$ is not required to be slowly varying.
\end{remark}

\section{Proof of Theorem~\ref{TH:MAIN}}

The proof of Theorem~\ref{TH:MAIN} follows exactly the same structure as that of \cite[Theorem 1.7]{Brasseur} but with Lemma~\ref{LE:main} instead of \cite[Lemma 3.8]{Brasseur}. For the convenience of the reader, we outline the main steps of the construction. Let $(\lambda_{j,k})_{j,k\in\N}$ be any sequence satisfying the conclusion of Lemma~\ref{LE:main}. Let $M\in\N$ be such that $M>s$. Set $C_M=2(M+2)$ and, for all $j,k\in\N$, let
$$ m_{j,k}=(C_M2^jj,\cdots,C_M2^jj,k)\in\N^N. $$
Set $u(t)=e^{-1/t^2}\mathds{1}_{(0,\infty)}(t)$ and $v(t)=u(1+t)u(1-t)$. In addition, set
$$ \psi(x)=\prod_{j=1}^N\frac{1}{2}\psi_0\left(\frac{x_j}{2}\right) \,\, \text{ where } \,\, \psi_0(t)=\frac{v(t)}{v(t-1)+v(t)+v(t+1)}. $$
Define
$$ f(x)=\sum_{j,k\in\N}\lambda_{j,k}\hspace{0.1em}2^{-j(s-\frac{N}{p})}\psi(2^jx-m_{j,k}). $$
Then, arguing as in (the proof of) \cite[Theorem 1.7]{Brasseur}, we have
$$ \|f\|_{B_{p,q}^s(\R^N)}\lesssim \bigg(\sum_{j=0}^\infty\bigg(\sum_{k=0}^\infty\lambda_{j,k}^p\bigg)^{\frac{q}{p}}\bigg)^\frac{1}{q}<\infty, $$
so that $f\in B_{p,q}^s(\R^N)$.
(This uses only the quarkonial/subatomic decomposition of $B_{p,q}^s(\R^N)$, see \cite{Trieb}.) On the other hand, we have
$$ \|f(\cdot,y)\|_{B_{p,\infty}^{(s,\Psi)}(\R^{N-1})}\gtrsim \sup_{j\in\N}\, 2^{\frac{j}{p}}\lambda_{j,\lfloor 2^jy\rfloor}\Psi(2^{-j})=\infty \quad\text{for a.e. }y\in[1,2], $$
so that $f(\cdot,y)\notin B_{p,\infty}^{(s,\Psi)}(\R^{N-1})$ for a.e. $y\in [1,2]$.
(This uses only the representation of $B_{p,\infty}^{(s,\Psi)}(\R^{N-1})$ via differences, as in Definition~\ref{G:BESOV}.)
This can be extended to a.e. $y\in\R$ (rather than ``a.e. $y\in[1,2]$") and to any dimension $d\in\N$ with $1\leq d<N$ using the same arguments as in \cite{Brasseur}.

\vspace{2mm}

\end{document}